\documentclass[letterpaper,12pt,reqno,tbtags]{amsart} 
\usepackage[portrait,margin=3cm]{geometry}  
\usepackage{mathrsfs}
\usepackage[
    colorlinks=true,%
    breaklinks,
    hyperindex,%
    plainpages=false,%
    bookmarksopen=false,%
    linkcolor=blue,%
    anchorcolor=blue,%
    citecolor=blue,%
    filecolor=black,%
    menucolor=black,%
    pagecolor=black,%
    urlcolor=blue,%
    pdfview=FitH,
    pdfstartview=FitH,
    hyperfigures,
    bookmarksnumbered%
  ]{hyperref} 
\usepackage{amsmath,amssymb,amsthm,amsfonts,amsbsy,latexsym,dsfont,tikz}

\usepackage{enumerate}

\newcommand{\R}[0]{\mathbb{R}}

\theoremstyle{remark}
\newtheorem{thm}{Theorem}
\newtheorem{lem}[thm]{Lemma}

\newtheorem{rem}[thm]{Remark}

     \let\gl=\lambda             
  
   \let\gL=\Lambda

\newcommand{\cH}{\mathcal{H}}

\newcommand{\cT}{\mathcal{T}}

\DeclareMathOperator{\E}{\mathds{E}}

\begin{document}

\title[Stein equation for weighted sums of independent $\chi^{2}$ distributions]{A note on Stein equation for weighted sums of independent $\chi^{2}$ distributions}
\author[Xiaohui Chen]{Xiaohui Chen}
\address{\newline Department of Statistics\newline
University of Illinois at Urbana-Champaign\newline
S. Wright Street, Champaign, IL 61820\newline
{\it E-mail}: \href{xhchen@illinois.edu}{\tt xhchen@illinois.edu}\newline 
{\it URL}: \href{http://publish.illinois.edu/xiaohuichen/}{\tt http://publish.illinois.edu/xiaohuichen/}
}
\author[Partha Dey]{Partha Dey}
\address{\newline Department of Mathematics\newline
University of Illinois at Urbana-Champaign\newline
S. Wright Street, Champaign, IL 61820\newline
{\it E-mail}: \href{psdey@illinois.edu}{\tt psdey@illinois.edu}\newline 
{\it URL}: \href{https://faculty.math.illinois.edu/~psdey/}{\tt https://faculty.math.illinois.edu/~psdey/}
}

\date{First version: \today}
\thanks{X. Chen's research is supported in part by NSF CAREER Award DMS-1752614 and a Simons Fellowship.}

\begin{abstract}
This note provides the Stein equation for weighted sums of independent $\chi^{2}$ distributions.
\end{abstract}
\maketitle

\section{Introduction}

Stein's method, first introduced in~\cite{stein1972}, is a powerful tool to bound the distance between a probability distribution and the Gaussian distribution. Over the past decades, Stein's method has been extended to other distributions including Poisson~\cite{chen1975}, exponential~\cite{Chatterjee11}, $\chi^{2}$~\cite{Gaunt2016}, and gamma~\cite{Luk1994}. At the core of Stein's method is a differential operator $\cT$, which generates a differential equation known as the {\it Stein equation} characterizing the target distribution $\pi$:
\begin{equation}
\label{eqn:Stein_eqn}
\cT f(x) = h(x) - \E_{Z \sim \pi}[h(Z)]
\end{equation}
for a collection of functions $h \in \cH$. On one hand, the Stein equation~\eqref{eqn:Stein_eqn} satisfies $\E_{Z \sim \pi}[\cT f(Z)] = 0$. On the other hand, if $\pi$ has an absolutely continuous density with suitable regularity conditions, then the Stein equation~\eqref{eqn:Stein_eqn} has a unique solution $f := f_{h}$ for any given piecewise continuous function $h$. Thus for any random variable $X \sim \nu$, if $\E_{X \sim \nu}[\cT f_{h}(X)] = 0$ over a rich class of functions $h \in \cH$, then $\nu = \pi$. Quantitatively, taking expectation on both sides of~\eqref{eqn:Stein_eqn} gives
\[
\E_{X \sim \nu}[\cT f_{h}(X)] = \E_{X \sim \nu}[h(X)] - \E_{Z \sim \pi}[h(Z)].
\]
In order to control the distance between $\nu$ and $\pi$, it is enough to estimate the quantity $\E_{X \sim \nu}[\cT f_{h}(X)]$ over $ h \in \cH$.

In this note, we derive a characterizing operator and the associated Stein equation for weighted sums of independent $\chi^{2}$ distributions. Such distributions arise as weak limits of degenerate $U$-statistics~\cite{Serfling1980}, which are useful in goodness-of-fit tests for distribution functions such as the Cram\'er-von Mises test statistic.

\section{Characterizing operator for $\chi^{2}$ distributions}
\label{sec:Stein_characterizing_operator_chisquares}

We first start with the characterizing operator for one $\chi^{2}$ distribution. Then we derive the general results for weighted sums of independent $\chi^{2}$ distributions.

\subsection{Characterizing operator for one $\chi^{2}$ distribution}

Let $Q\sim \chi^{2}_{p}$ and $\widetilde{Q} = Q-p$. Then the operator
\[ \widetilde{\cT}f(x)=2(x+p)f'(x)-xf(x) \]
is a characterizing operator for the centered $\chi^{2}_{p}$ distribution, in the sense that $\widetilde{Q}\sim \chi^{2}_{p}-p$ if and only if $\E(\widetilde{\cT}f(\widetilde{Q}))=0$ for all ``smooth'' function $f$. To prove that $\widetilde{Q}\sim \chi^{2}_{p}-p$ distribution indeed satisfies $\E(\widetilde{\cT}f(\widetilde{Q}))=0$ for all ``smooth'' function $f$, we use the following lemma.

\begin{lem}[Integration by parts formula]
\label{lem:gaussian_IBP_chisquare}
If $Q\sim \chi^{2}_{p}$, then 
\[ \E\left((Q-p)f(Q)\right) = \E\left( 2Qf'(Q)\right) \]
holds for any absolutely continuous function $f : \R \to \R$ such that the expectations $\E |f(Q)|, \E|Qf(Q)|, \E |Qf'(Q)|$ are finite. Equivalently, if $\widetilde{Q} = Q-p$, then 
\[ \E\left(\widetilde{Q}f(\widetilde{Q})\right) = \E\left( 2(\widetilde{Q}+p)f'(\widetilde{Q})\right). \]
\end{lem}

\begin{proof}[Proof of Lemma \ref{lem:gaussian_IBP_chisquare}]
Note that we can write $Q=\sum_{i=1}^{p}Z_{i}^{2}$, where $Z_{1},Z_2,\ldots,Z_p$ are i.i.d.~standard Gaussian random variables and the equality holds in distribution. Then by the Gaussian integration by parts, we have 
\begin{align*}
	\E\left((Q-p)f(Q)\right) = \sum_{i=1}^{p} \E\left((Z_{i}^{2}-1)f(Q)\right) &= \sum_{i=1}^{p} \E\left(Z_i \frac{\partial}{\partial Z_i}f(Q)\right) \\
	&= \sum_{i=1}^{p} \E\left(2Z_{i}^{2}f'(Q)\right) = \E\left( 2Qf'(Q)\right). 
\end{align*}
The second part is an immediate consequence of the first part. 
\end{proof}

\subsection{Some combinatorial results}

Given a sequence of distinct non-zero real numbers $\gl_{1},\gl_{2},\ldots,\gl_{r}$, we define 
\begin{align*}
	\gL_{k,i} &:= \sum_{S\subseteq [r]\setminus\{i\}, |S|=k}\ \prod_{j\in S} \gl_{j} \text{ and }\\
	\gL_{k}&:= \sum_{S\subseteq [r], |S|=k}\ \prod_{j\in S} \gl_{j} 
\end{align*}
for $i, k\in [r]$, where $[r] = \{1,\dots,r\}$. Define $\gL_{0}\equiv\gL_{0,i}\equiv 1$. Clearly $\gL_{r,i} = 0$ for all $i \in [r]$.

\begin{lem}
\label{lem:combinatorial}
For $i,k \in [r]$, we have 
\begin{align*}
	\gL_{k}-\gL_{k-1,i}\gl_{i} = \gL_{k,i}  \text{ and } \sum_{i=1}^{r} \gL_{k,i}\gl_{i} &= (k+1)\gL_{k+1}.
\end{align*}
\end{lem}

\begin{proof}[Proof of Lemma \ref{lem:combinatorial}]
The first claim follows from the definitions of $\gL_{k,i}$ and $\gL_{k}$. Note that $\gL_{k}$ involves the summation of the product terms of $k$ distinct $\lambda_{i}$'s, which implies that $\sum_{i=1}^{r} \gL_{k,i} = (r-k) \gL_{k}$. Thus we have 
\[
\sum_{i=1}^{r} \gL_{k,i}\gl_{i} = \sum_{i=1}^{r} (\gL_{k+1} - \gL_{k+1,i}) = r \gL_{k+1} - \sum_{i=1}^{r} \gL_{k+1,i} = (k+1) \gL_{k+1}.
\]
\end{proof}

\subsection{Stein equation for weighted sums of independent $\chi^{2}$ distributions}

Let $Q_{i}\sim \chi^{2}_{m_i}, i=1,2,\ldots,r$ be independent chi-squared random variables and $\gl_{1},\gl_{2},\ldots,\gl_{r}$ be a sequence of distinct non-zero real numbers. We consider the random variable
\[ U=\sum_{i=1}^{r}\gl_{i}Q_{i}, \]
which is a weighted sum of independent $\chi^{2}$ random variables.
Define $\mu=\E(U)= \sum_{i=1}^r \gl_{i} m_{i}, \widetilde{U}:=U-\mu, \widetilde{Q_{i}}:=Q_{i}-m_{i}, i=1,2,\ldots,r$. We also define 
\begin{align*}
	\mu_{k}:= \sum_{i=1}^r \gl_{i}^2 \gL_{k-1,i} m_{i} \quad \text{ for } k\ge 1 
\end{align*}
and $\mu_{0}=0$. The main result of this note is the following Stein equation for $\widetilde{U}$.

\begin{thm}[Stein equation for $\widetilde{U}$]
\label{thm:Stein_characterizing_operator_centered}
Let $f : \R \to \R$ be an $r$-th differentiable function such that $\E|f^{(k)}(U)|$ and $\E|U f^{(k)}(U)|, k = 0,1,\dots,r$ are finite. Then we have $\E \widetilde{\cT}f(\widetilde{U}) = 0$, where 
\begin{equation}
\label{eqn:Stein_characterizing_operator_centered}
\widetilde{\cT}f(x) = \sum_{k=0}^{r} (-2)^{k} \big(\mu_{k} +\gL_{k} x\big) f^{(k)}(x).
\end{equation}
\end{thm}

\begin{rem}
Stein equation for the non-centered weighted sum $U$ of independent $\chi^{2}$ distributions is given by $\E \cT f(U)=0$, where 
\begin{equation}
\label{eqn:Stein_characterizing_operator_noncentered}
\cT f(x) = \sum_{k=0}^{r} (-2)^{k} \big(\mu_{k} + \gL_{k} x -\gL_{k}\mu\big) f^{(k)}(x).
\end{equation}
\end{rem}

\begin{proof}[Proof of Theorem \ref{thm:Stein_characterizing_operator_centered}]
Take a smooth function $f$. By the integration by parts formula for $\chi^{2}$ distribution in Lemma~\ref{lem:gaussian_IBP_chisquare}, we have 
\begin{align*}
	\E\left (\widetilde{U}f(\widetilde{U}) \right) &= 2\sum_{i=1}^{r} \gl_{i}^{2} \E\left ( Q_{i}f^{(1)}(\widetilde{U}) \right)\\
	&= 2\sum_{i=1}^{r} \gl_{i}^{2}m_{i} \E\left ( f^{(1)}(\widetilde{U}) \right)+ 2\sum_{i=1}^{r} \gl_{i}^{2} \E\left ( \widetilde{Q}_{i}f^{(1)}(\widetilde{U}) \right)\\
	&= 2\mu_1 \E\left ( f^{(1)}(\widetilde{U}) \right)+ 2\gL_{1} \E\left (\widetilde{U} f^{(1)}(\widetilde{U}) \right) - 2\sum_{i=1}^{r} \gl_{i}\gL_{1,i} \E\left ( \widetilde{Q}_{i}f^{(1)}(\widetilde{U}) \right)\\
	&= 2 \E\left ( (\mu_1 + \gL_{1} \widetilde{U}) f^{(1)}(\widetilde{U}) \right) - 2^{2}\sum_{i=1}^{r} \gl_{i}^{2}\gL_{1,i} \E\left ( {Q}_{i}f^{(2)}(\widetilde{U}) \right),
\end{align*}
where the third equality follows from Lemma \ref{lem:combinatorial}. Expanding the last term we have 
\begin{align*}
	&\sum_{i=1}^{r} \gl_{i}^{2}\gL_{1,i} \E\left ( {Q}_{i}f^{(2)}(\widetilde{U}) \right)\\
	&= \sum_{i=1}^{r} \gl_{i}^{2}\gL_{1,i}m_{i} \E\left ( f^{(2)}(\widetilde{U}) \right) + \sum_{i=1}^{r} \gl_{i}^{2}\gL_{1,i} \E\left ( \widetilde{Q}_{i}f^{(2)}(\widetilde{U}) \right)\\
	&= \mu_{2} \E\left ( f^{(2)}(\widetilde{U}) \right) + \gL_{2} \E\left (\widetilde{U} f^{(2)}(\widetilde{U}) \right) - 2\sum_{i=1}^{r} \gl_{i}\gL_{2,i} \E\left ( \widetilde{Q}_{i}f^{(2)}(\widetilde{U}) \right)\\
	&= \mu_{2}\E\left ( (\mu_2 + \gL_{2} \widetilde{U}) f^{(2)}(\widetilde{U}) \right) - 2\sum_{i=1}^{r} \gl_{i}^{2}\gL_{2,i} \E\left ( {Q}_{i}f^{(3)}(\widetilde{U}) \right). 
\end{align*}
Using induction we finally get 
\begin{align*}
	&\E\left (\widetilde{U}f(\widetilde{U}) \right)\\
	&= \sum_{k=1}^{r-1} (-1)^{k-1} 2^{k} \E\left ( \mu_{k} f^{(k)}(\widetilde{U})+\gL_{k} \widetilde{U} f^{(k)}(\widetilde{U}) \right) +(-1)^{r-1} 2^{r} \sum_{i=1}^{r} \gl_{i}^{2}\gL_{r-1,i} \E\left ( {Q}_{i}f^{(r)}(\widetilde{U}) \right)\\
	&= \sum_{k=1}^{r-1} (-1)^{k-1} 2^{k} \E\left( (\mu_{k} +\gL_{k} \widetilde{U}) f^{(k)}(\widetilde{U}) \right) +(-1)^{r-1} 2^{r} \gL_{r}\E\left ((\mu+\widetilde{U})f^{(r)}(\widetilde{U}) \right),
\end{align*}
where the last step follows from Lemma \ref{lem:combinatorial} and $\gL_{r,i} = 0$. Thus $\widetilde{U}$ satisfies the relation $\E \widetilde{\cT}f(\widetilde{U})=0$, where 
\begin{align*}
\widetilde{\cT}f(x) &= xf(x) + \sum_{k=1}^{r-1} (-2)^{k} \big(\mu_{k} +\gL_{k} x\big) f^{(k)}(x) + (-2)^{r} \gL_{r} (\mu+x) f^{(r)}(x).
\end{align*}
Then (\ref{eqn:Stein_characterizing_operator_centered}) follows from the last identity together with $\mu_{0}=0, \gL_{0} = 1$, and $\mu_{r} = \gL_{r} \mu$.
\end{proof}

\bibliographystyle{plain}
\bibliography{stein_chisquares}

\end{document}